\documentclass[10pt]{article}
\font\smallit=cmti10

\usepackage{mathptm}    
\usepackage{latexsym}   
\usepackage{amssymb}    
\usepackage{amsmath}    
\usepackage{amsthm}     
\usepackage{graphicx}
\usepackage{multicol}

\newtheorem{theorem}{Theorem}[section]
\newtheorem{lemma}[theorem]{Lemma}

\newenvironment{prf}[1]{\trivlist
\item[\hskip
\labelsep{\it #1.\hspace*{.3em}}]}{
\endtrivlist}

\newtheorem{predefinition}[theorem]{Definition}

\newtheorem{preremark}[theorem]{Remark}

\newtheorem{prenotation}[theorem]{Notation}

\newtheorem{preexample}[theorem]{Example}
\newenvironment{example}{\begin{preexample}\rm}{\end{preexample}}
\newtheorem{preclaim}[theorem]{Claim}

\newtheorem{prequestion}[theorem]{Question}

 \makeatletter
\def\emppsubsection{\@startsection{subsection}{2}{\z@}{-3.25ex plus -1ex minus -.2ex}{-1em}}

\makeatother

\newcommand \CA {{\mathcal A}}
\newcommand \ZZ {{\mathbb Z}}

\newcommand \PP {{\mathbb P}^1}
\newcommand \xx {{\bold x}}
\newcommand \g {{\bold g}}
\newcommand \yy {{\bold y}}

\newcommand \fg {{\mathfrak g}}
\newcommand \aaa {{\mathfrak a}}

\begin{document}

\begin{center}
\uppercase{\bf Compositions of Integers With Bounded Parts}
\vskip 20pt
{\bf Darren B Glass}\\
{\smallit Department of Mathematics, Gettysburg College, Gettysburg PA 17325}\\
{\tt dglass@gettysburg.edu}\\
\end{center}
\vskip 30pt

\begin{abstract}
In this note, we consider ordered partitions of integers such that each entry is no more than a fixed portion of the sum.  We give a method for constructing all such compositions as well as both an explicit formula and a generating function describing the number of $k$-tuples whose entries are bounded in this way and sum to a fixed value $g$.
\end{abstract}

\section{Introduction}

Assume you have $100$ pieces of candy that you want to split among your $3$ children, but you have a restriction: you don't want the oldest child to get more than one-third of the candy, you don't want the middle child to get more than two-fifths of the candy and you don't want the youngest child to get more than two-sevenths of the candy.  How many ways can you split the candy between the three children?  In this note, we address this question and the more general question of counting compositions (also known as ordered partitions) of integers so that no part is more than a fixed portion of the total.

More precisely, let us fix an integer $k$ and for each $1 \le i \le k$ let $\alpha_i$ be a rational number so that the sum of any $k-1$ of the $\alpha_i$ is at most $1$ but that all $k$ of the $\alpha_i$ add up to more than one.  (Note that if the $\alpha_i$ sum to exactly one or less than one then the question is trivial).  We wish to count the number of ordered $k$-tuples of integers $[g_1,g_2,\ldots,g_k]$ with $0 \le g_i \le \alpha_i \sum_{j=1}^k g_j$ for each $i$.  We will call such a composition $\alpha$-communal and our goal is to understand the structure of the set of $\alpha$-communal compositions.

As is often the case with counting questions, one might reasonably ask for either an explicit formula for $f(g)$, the number of $\alpha$-communal compositions of a given integer $g$, or for a nice closed form of the generating function defined by $F(x)=\sum f(g)x^g$.  In this note we answer both questions: in particular, Section \ref{S:explicit} describes an explicit, if ugly, formula for $f(g)$.  In section \ref{S:GenFn} we show that the set of $\alpha$-communal compositions forms a monoid and we describe an explicit structure for this monoid. These results lead to Theorem \ref{T:gen}, in which we give a closed form description of $F(x)$.

We close our note with a section giving examples of the results in the previous sections.  One special case we consider is the case where all of the $\alpha_i$ are equal to $\frac{1}{k-1}$.  This situation was studied by the author in \cite{Gint}, where we referred to the $k$-tuples simply as communal compositions.  Example \ref{E:Andrews} shows that some of the results of that paper are special cases of this results in this note.

We note that the problem we are considering in this note is related, but not identical, to the ``17 Horses Puzzle", in which three sons are supposed to split 17 horses so that one son gets $\frac{1}{2}$, one son gets $\frac{1}{3}$ and the final son gets $\frac{1}{9}$ of the horses.  A fuller discussion of this problem attributing it to Tartaglia in the sixteenth century can be found in \cite[Prob 2.11]{P}.

\section{Combinatorial Formula}\label{S:explicit}

In this section, we give an explicit formula for $f(g)$, the number of $\alpha$-communal $k$-tuples summing to $g$, for any fixed integer $g$.  In particular, let $\{\alpha_i\}_{i=1}^k$ be a set of rational numbers so that the sum of any $k-1$ of these numbers is less than one but the sum of all $k$ is at least one.  We wish to count the number of ordered $k$-tuples of integers $[g_1,\ldots,g_k]$ so that $\sum g_i = g$ and $0 \le g_i \le \alpha_i g $ for all $i$.

Given an $\alpha$-communal $k$-tuple $[g_1,\ldots,g_k]$ with $\sum g_i = g$, we set $\epsilon_i = \lfloor \alpha_i g \rfloor - g_i$.  The condition that our $k$-tuple is $\alpha$-communal implies that each $\epsilon_i \in \ZZ_{\ge 0}$. Moreover, one can see that $\sum_i \epsilon_i = \sum_i \lfloor \alpha_i g \rfloor - g$, which we will denote by $s_g$.

Conversely, given a set of $k$ nonnegative integers $\{\epsilon_i\}$ so that $\sum \epsilon_i = s_g$, we set $g_i = \lfloor \alpha_i g \rfloor - \epsilon_i$.  It is clear that $g_i \le \lfloor \alpha_i g \rfloor$.  On the other hand, our hypotheses on the $\alpha_i$ include the fact that for any fixed $j$ we have that $\sum_{i \ne j} \alpha_i \le 1$.  Thus, we compute:

\begin{eqnarray*}
\sum_{i \ne j} \alpha_i & \le & 1 \\
\sum_{i \ne j} \alpha_i g & \le & g \\
\sum_{i \ne j} \lfloor\alpha_i g \rfloor & \le & g \\
\sum_{i} \lfloor \alpha_i g \rfloor - g& \le & \lfloor \alpha_j g \rfloor \\
s_g &\le& \lfloor \alpha_j g \rfloor \\
\end{eqnarray*}

\noindent which implies that $\epsilon_j \le \lfloor \alpha_i g \rfloor$ and therefore that each $g_j \ge 0$. Moreover, one can check that $\sum g_i = g$, implying that $[g_1,\ldots,g_k]$ is an $\alpha$-communal $k$-tuple summing to $g$.

In particular, there is a natural bijection between the set of $\alpha$-communal $k$-tuples summing to $g$ and the number of $k$-tuples of nonnegative integers summing to $s_g$. In order to count these, we will use the following lemma, which is standard in combinatorial number theory.  See \cite[Prop 21.5]{BB} for one proof.

\begin{lemma}
The number of solutions to the equation $x_1+\ldots+x_k = m$ where all of the $x_i$ are nonnegative integers is equal to $\binom{m+k-1}{k-1}$.
\end{lemma}

The following theorem is an immediate consequence.

\begin{theorem}\label{T:exp}
The number of $\alpha$-communal $k$-tuples whose entries sum to $g$ is given by
\[f(g) = \left(\begin{array}{c}
                    \sum_{i=1}^k \lfloor \alpha_i g \rfloor - g + k - 1 \\
                    k-1
                    \end{array}
                    \right) \]
\end{theorem}

If the rational number $\alpha_i$ can be expressed in lowest terms as $\frac{m_i}{n_i}$ and we set $n$ to be the least common multiple of the $n_i$ then we note that the formula in Theorem \ref{T:exp} can be instead expressed as a collection of $n$ polynomials of degree $k-1$ depending on the value of $g$ mod $n$. We return to this formula in explicit examples in section \ref{S:ex}.

\section{Structure of $\alpha$-communal Compositions}\label{S:GenFn}

To begin, let us fix a $k$-tuple $(\alpha_1,\ldots,\alpha_k)$.   We leave the proof of the following lemma to the reader:

\begin{lemma}\label{L:sum}
If $\xx = [x_1,\ldots,x_k]$ and $\yy = [y_1,\ldots,y_k]$ are $\alpha$-communal $k$-tuples then so is their sum $\xx + \yy = [x_1+y_1,\ldots,x_k+y_k]$.
\end{lemma}

In particular, the set of $\alpha$-communal $k$-tuples forms a submonoid of the additive monoid $\ZZ_{\ge 0}^k$.  This leads to the natural question of finding a set of generators for the set.  In order to do so, let us first introduce some notation.  We write the rational number $\alpha_i$ as the fraction $\frac{m_i}{n_i}$ in lowest terms and define $N$ to be the product $\prod_{i=1}^k n_i$. Additionally, we set $A = N(\sum_{i=1}^k \alpha_i - 1)$ and $\hat{\alpha_i} = 1 - \sum_{j \ne i}\alpha_j$.  For each $i$, let us define the $k$-tuple $\xx_i$ as follows:

\[\xx_i=\frac{N}{n_i}[\alpha_1,\ldots,\alpha_{i-1},\hat{\alpha_i},\alpha_{i+1},\ldots,\alpha_k]\]

We note that the entries of each $\xx_i$ are nonnegative integers because of the assumption that the sum of any $k-1$ of the $\alpha_i$ is at most $1$.  Moreover, it is an easy exercise to check that each $\xx_i$ is $\alpha$-communal, and by Lemma \ref{L:sum} every triple obtained as a nonnegative integral linear combination of the $\xx_i$ will be as well.

\begin{lemma}
Let $\fg = [g_1,\ldots,g_k]$ be an $\alpha$-communal $k$-tuple with $g = \sum g_i$. Then one can write $\fg$ as the sum of the $\xx_i$ in the following way:

\[\fg = \sum_{j=1}^k \frac{m_jg-n_jg_j}{A}\xx_j\]

\end{lemma}

\begin{proof}
Consider the $i^{th}$ coordinate of the $k$-tuple defined as $\sum_{j=1}^k (m_jg-n_jg_j)\xx_j$.  In particular, we can compute that it will be

\begin{eqnarray*}
\sum_{j \ne i}(m_jg-n_jg_j)\frac{Nm_j}{n_i n_j}+(m_ig-n_ig_i)\left(\frac{N}{n_i}-\sum_{j\ne i}\frac{Nm_j}{n_in_j}\right)&=&\alpha_iNg - \sum_{j \ne i}\alpha_iNg_j - Ng_i + \sum_{j \ne i} \alpha_jNg_i\\
&=&Ng_i(\sum_{j=1}^k \alpha_j -1)\\
&=&Ag_i
\end{eqnarray*}

The lemma immediately follows.
\end{proof}

\noindent By assumption, $m_jg \ge n_jg_j$ for each $j$, so the coefficients are all nonnegative.  In particular, if each $m_jg - n_jg_j$ is a multiple of $A$ then we have shown that one can write the $\fg$ as an integral combination of the $\xx_i$.  In particular, if $A = 1$  then the $\xx_i$ form a basis for the monoid of $\alpha$-communal $k$-tuples, a situation which we explore in Example \ref{Ex:A1}.  In the case where $A>1$ we will not get all $k$-tuples in this manner.  To cover this case, let $a_j$ be the least residue of $m_jg - n_jg_j$ mod $A$.  Then it follows that $(m_jg-n_jg_j-a_j)/A$ is a nonnegative integer and we compute:

\begin{eqnarray*}
\sum_{i=1}^k\frac{m_ig-n_ig_i-a_i}{A}\xx_i &=&[g_1,\ldots,g_k]-\sum_{i=1}^k\frac{a_i}{A}\xx_i \\
&=&[g_1,\ldots,g_k]-[b_1,\ldots,b_k]
\end{eqnarray*}

\noindent where

\[b_j = \frac{1}{A}\left(N\hat{\alpha_j}\frac{a_j}{n_j}+ N \alpha_j \sum_{i \ne j} \frac{a_i}{n_i}\right)\]

In particular, we can write any $\alpha$-communal $k$-tuple $\g = [g_1,\ldots,g_k]$ in a unique way as the sum of a `base' $k$-tuple $[b_1,\ldots,b_k]$ and a nonnegative integral combination of the $\xx_i$.  Moreover, because the $k$-tuples $\g$ and $\xx_i$ consist of integers it must be the case that the $b_i$ are all integers and therefore the base $k$-tuples that we need to consider are exactly those arising from $k$-tuples $(a_1,\ldots,a_k)$ of least residues which make them integers.  In particular, they will be the $k$-tuples in the set:
\[\CA = \left\{(a_1,\ldots,a_k) \, \begin{array}{|l}
0 \le a_i < A\\
N\hat{\alpha_j}\frac{a_j}{n_j}+ N \alpha_j \sum_{i \ne j} \frac{a_i}{n_i} \equiv 0 \text{ mod } A \text{ for all } 1 \le j \le k\\
\end{array} \right\}\]

\noindent which will have at most $A^k$ elements and for `generic' choices of the $\alpha_i$ will have $A^{k-1}$ elements -- if the $m_i$ and $n_i$ are all relatively prime to $A$ then one deduces that the congruence conditions are in fact equivalent and allow one to get an explicit formula for $a_k$ in terms of the other $a_i$.

For each $\aaa \in \CA$, we define $b(\aaa)$ to be the sum of the entries in the corresponding base $k$-tuple, and we can compute:

\begin{eqnarray*}
b(\aaa) &=& \sum_{j=1}^k b_j \\
&=& \sum_{j=1}^k \frac{N}{A}\left(\hat{\alpha_j}\frac{a_j}{n_j}+ \alpha_j \sum_{i \ne j} \frac{a_i}{n_i}\right) \\
&=&\frac{N}{A}\sum_{j=1}^k \hat{\alpha_j}\frac{a_j}{n_j} + \frac{N}{A}\sum_{i \ne j} \frac{\alpha_ja_i}{n_i} \\
&=&\frac{N}{A}\left(\sum_{i=1}^k \frac{a_i}{n_i}\right)
\end{eqnarray*}

It follows that the number of $\alpha$-communal $k$-tuples summing to $g$ is the same as the number of ways to write $g$ as the sum of a number of the form $b(\aaa)$ for some $k$-tuple $(a_1,\ldots,a_k) \in \CA$ and a nonnegative integral linear combination of the numbers $\frac{N}{n_i}$.  Theorem \ref{T:gen} is an immediate consequence using basic facts on generating functions (see \cite{AE} or \cite{W}, for example).

\begin{theorem}\label{T:gen}
Let $f(g)$ be the number of $k$-tuples of nonnegative integers $g_1,\ldots,g_k$ so that $\sum g_i=g$ and $g_i \le \alpha_ig$ where the $\alpha_i$ are rational numbers as described as above.  Then the function $f(g)$ can be described by a generating function in the following way:
\[F(x) = \sum_{g=0}^{\infty}f(g)x^g = \frac{\sum_{\aaa \in \CA} x^{b(\aaa)}}{\prod_{i=1}^k(1-x^{N/n_i})}\]

\end{theorem}

\section{Examples}\label{S:ex}

Computing explicit formulas from Theorem \ref{T:gen} can be difficult in general, but in many specific cases, such as when the $m_i$ and $n_i$ share common factors, the terms reduce greatly and it is not difficult to compute $F(x)$.  We close this note with some examples and applications.

\begin{example}\label{Ex:A1}
Let $k=3$ with $\alpha_1=\frac{1}{2}, \alpha_2 = \frac{1}{3}, \alpha_3=\frac{1}{5}$.  In particular, $\alpha_1+\alpha_2+\alpha_3=\frac{31}{30}$ so we see that $A=1$.  As discussed in the previous section, this implies that every $\alpha$-communal triple can be written as a nonnegative integral combination of the triples $\xx_1=[7,5,3], \xx_2=[5,3,2], \xx_3=[3,2,1]$.  In particular, the generating function whose coefficients give us the number of $\alpha$-communal triples summing to $g$ is given by $F(x) = ((1-x^{15})(1-x^{10})(1-x^6))^{-1}$.  At the same time, Theorem \ref{T:exp} tells us that a formula for the number of $\alpha$-communal triples summing to $g$ is given by
\[f(n) = \left(\begin{array}{c}
                    \lfloor \frac{g}{2} \rfloor + \lfloor \frac{g}{3} \rfloor + \lfloor\frac{g}{5} \rfloor - g + 2 \\
                    2
                    \end{array}
                    \right) \]

Other examples of $k$-tuples so that the term $A=1$, making it particularly easy to write down the structure of $\alpha$-communal compositions, include $\alpha = (\frac{2}{5},\frac{1}{17},\frac{7}{23},\frac{9}{38}), (\frac{1}{4},\frac{1}{7},\frac{2}{11},\frac{4}{13},\frac{2}{17}),$ and $(\frac{3}{11},\frac{3}{13},\frac{2}{15},\frac{1}{17},\frac{7}{23})$
\end{example}

\begin{example} \label{E:Andrews}
We next wish to apply our results to the classical problem of counting triangles with a fixed perimeter and integer sides, as considered by Andrews in \cite{A}.  In our context, this is the case where $k=3$ and each of the $\alpha_i = \frac{1}{2}$, so each $m_i=1$ and each $n_i=2$.  We wish to consider the more general situation where we set $\alpha_i = \frac{1}{k-1}$ for each $1 \le i \le k$.

Let $\ell$ be the least residue of $g$ mod $k-1$.  In particular we have $\lfloor \frac{g}{k-1} \rfloor = \frac{g-\ell}{k-1} \in \ZZ$.  It follows from Theorem \ref{T:exp} that the number of $\alpha$-communal $k$-tuples summing to $g$ is given by the function

\[f(g) = \left(\begin{array}{c}
                    \frac{g-\ell k}{k-1} + k - 1 \\
                    k-1
                    \end{array}
                    \right) \]

As an illustration, in the case $k=3$ this reduces to the formula $f(g) = \frac{1}{8}(g^2+6g+8)$ if $g$ is even and $\frac{1}{8}(g^2-1)$ if $g$ is odd.

If we instead wish to find the generating function describing the sequence $\{f(g)\}$, we note that in the notation of Section \ref{S:GenFn} we can compute that $A=(k-1)^{k-1}$ and the set $\CA$ consists of all $k$-tuples $(a_1,\ldots a_k)$ so that $0 \le a_i <(k-1)^{k-1}$ and all of the $a_i$ are congruent mod $k-1$.  In particular, we can write the set $\CA$ as a disjoint union of sets $\CA_s$ for $0 \le s \le (k-2)$ where all of the $a_i$ are congruent to $s$ mod $k-1$.  We note that if $\aaa \in \CA_s$ then we have that $b(\aaa) = \sum a_i \equiv s$ mod $k-1$ as well.

We observe that for any $g \equiv 0$ mod $k-1$, the number of ways to write it as a sum of $k$ numbers which are multiples of $k-1$ is a straightforward thing to compute, and standard results about generating functions imply that ${\displaystyle \sum_{\aaa \in \CA_0} x^{\sum a_i} = (1+x^{k-1}+x^{2(k-1)}+\ldots+x^{(k-2)(k-1)})^k}$. Similarly, there is a bijection between $k$-tuples in $\CA_s$ whose entries sum to $g$ and $k$-tuples in $\CA_0$ whose entries sum to $g-ks$, allowing one to compute that ${\displaystyle \sum_{\aaa \in \CA_s} x^{b(\aaa)} = x^{ks}(1+x^{k-1}+\ldots+x^{(k-2)(k-1)})^k}$.  We now use Theorem \ref{T:gen} to compute the generating function explicitly:

\begin{eqnarray*}
F(x) &=&  \frac{\sum_{\aaa \in \CA} x^{b(\aaa)}}{\prod_{i=1}^k(1-x^{N/n_i})}\\
&=& \frac{(1+x^k+\ldots+x^{k(k-2)})(1+x^{k-1}+x^{2(k-1)}+\ldots+x^{(k-2)(k-1)})^k}{(1-x^{(k-1)^2})^k}\\
&=&\frac{(1-x^{k(k-1)})(1-x^{(k-1)^2})^k}{(1-x^k)(1-x^{k-1})^k(1-x^{(k-1)^2})^k}\\
&=&\frac{1-x^{k(k-1)}}{(1-x^k)(1-x^{k-1})^k}
\end{eqnarray*}

\noindent which agrees with the formula given in \cite[Thm 10]{Gint}.
\end{example}

\begin{example}
Let $k=3$ and $m_i = 1$, and $n_1=n_2=2$ with $n_3 =n \ge 2$.  In this case, $A=4$ and one can see that $(a_1,a_2,a_3) \in \CA$ if and only if $a_1 \equiv a_2$ (mod $2$) and $a_3 \equiv a_1 + \frac{a_1+a_2}{2}n$ (mod $2$).  In particular, if $n$ is odd then $\CA$ consists of the sixteen triples:

\[\begin{array}{cccc}
(0,0,0) & (0,0,2) & (2,2,0) & (2,2,2)\\
(0,2,1) & (0,2,3) & (2,0,1) & (2,0,3)\\
(1,1,0) & (1,1,2) & (3,3,0) & (3,3,2)\\
(1,3,1) & (1,3,3) & (3,1,1) & (3,1,3)\\
\end{array}\]

One can then use Theorem \ref{T:gen} to compute that the generating function in this case is:
\begin{eqnarray*}
F(x) &=&\frac{1+x^2+x^n+2x^{n+1}+x^{n+2}+2x^{n+3}+x^{2n}+2x^{2n+1}+x^{2n+2}+2x^{2n+3}+x^{3n}+x^{3n+2}}{(1-x^4)(1-x^{2n})^2}\\
& =& \frac{1+2x^{n+1}+x^{2n}}{(1-x^2)(1-x^n)(1-x^{2n})}
\end{eqnarray*}

Similarly, if $n$ is even one can show that the generating function simplifies to
\[F(x) = \frac{1+x^{n+1}}{(1-x^2)(1-x^n)^2}\]

One can additionally use Theorem \ref{T:exp} to show that the number of $\alpha$-communal triples summing to $g$ is given by $f(g) = \frac{1}{2}(\lfloor\frac{g}{n}\rfloor +1)(\lfloor\frac{g}{n}\rfloor + \epsilon_g)$
\noindent where $\epsilon_g = 2$ if $g$ is even and $\epsilon_g = 0$ if $g$ is odd.

It is worth noting that this question was one of the original motivations for this note, as it is related to the question of counting the irreducible components of the moduli space of dihedral covers of the projective line.  While we will not go into details here, the interested reader might consult \cite[\S 5]{Gint} and \cite{GP} for the similar problem comparing $(\frac{1}{2},\frac{1}{2},\frac{1}{2})$-communal triples to the moduli space of $(\ZZ/2\ZZ)^2$-covers of $\PP$.
\end{example}

\begin{example}
Returning to the example from the opening paragraph of this note, let us let $\alpha_1 = \frac{1}{3}, \alpha_2=\frac{2}{5}$ and $\alpha_3=\frac{2}{7}$.  Then $A = 2$ and $\CA$ consists of triples $(a_1,a_2,a_3)$ with either one or three entries equalling $0$ and the others equalling $1$.  In particular, we see that any division of candy that satisfies our restrictions can be written as the sum of one of the triples in the set $\{[0,0,0],[6,7,5],[8,10,7], [9,11,8]\}$ and an integral linear combination of the triples $[5,6,4], [7,8,6], [11,14,10]$. Moreover, the generating function associated to this problem is $F(x) = \frac{1+x^{18}+x^{25}+x^{28}}{(1-x^{15})(1-x^{21})(1-x^{35})}$.  A computer algebra system will now tell us that the expansion of this as a power series includes the terms
\[F(x) = \ldots x^{97} + 3 x^{98} + 3 x^{99} + 3 x^{100} + x^{101} + 3 x^{102} + 3 x^{103}+\ldots\]
which tells us that there are three ways to distribute $100$ pieces of candy according to these rules (in particular, they are $[32,40,28],[33,39,28]$, and $[33,40,27]$) but a unique way of dividing $101$ pieces, $[33,40,28]$.
\end{example}

\bibliographystyle{amsplain}
\bibliography{partitions}

\providecommand{\bysame}{\leavevmode\hbox to3em{\hrulefill}\thinspace}
\providecommand{\MR}{\relax\ifhmode\unskip\space\fi MR }
\providecommand{\MRhref}[2]{%
  \href{http://www.ams.org/mathscinet-getitem?mr=#1}{#2}
}
\providecommand{\href}[2]{#2}
\begin{thebibliography}{1}

\bibitem{A}
George~E. Andrews, \emph{A note on partitions and triangles with integer
  sides}, Amer. Math. Monthly \textbf{86} (1979), no.~6, 477--478. \MR{533570
  (80g:10009)}

\bibitem{AE}
George~E. Andrews and Kimmo Eriksson, \emph{Integer partitions}, Cambridge
  University Press, Cambridge, 2004. \MR{2122332 (2006b:11125)}

\bibitem{BB}
B\'ela Bajnok, \emph{An invitation to abstract mathematics}, Springer Verlag,
  New York, 2013.

\bibitem{GP}
Darren Glass and Rachel Pries, \emph{Hyperelliptic curves with prescribed
  {$p$}-torsion}, Manuscripta Math. \textbf{117} (2005), no.~3, 299--317.
  \MR{MR2154252 (2006e:14039)}

\bibitem{Gint}
Darren~B. Glass, \emph{Communal partitions of integers}, Integers \textbf{12}
  (2012), no.~3, 405--416. \MR{2955522}

\bibitem{P}
Miodrag~S. Petkovi{\'c}, \emph{Famous puzzles of great mathematicians},
  American Mathematical Society, Providence, RI, 2009. \MR{2541498
  (2010h:00008)}

\bibitem{W}
Herbert~S. Wilf, \emph{generatingfunctionology}, third ed., A K Peters Ltd.,
  Wellesley, MA, 2006. \MR{2172781 (2006i:05014)}

\end{thebibliography}

\end{document}